\documentclass[12pt]{article}
\usepackage{amsmath, amssymb, amsthm, amsfonts}
\usepackage{indentfirst}

\numberwithin{equation}{section}

\theoremstyle{plain}
\newtheorem{theorem}{Theorem}\numberwithin{theorem}{section}
\newtheorem{lemma}{Lemma}\numberwithin{lemma}{section}
\newtheorem{proposition}{Proposition}\numberwithin{proposition}{section}
\newtheorem{corollary}{Corollary}\numberwithin{corollary}{section}
\theoremstyle{definition}
\newtheorem{definition}{Definition}\numberwithin{definition}{section}
\theoremstyle{remark}
\newtheorem{remark}{Remark}\numberwithin{remark}{section}

\newcommand{\R}{\mathbb{R}}
\newcommand{\N}{\mathbb{N}}

\newcommand{\M}{\mathcal{M}}
\newcommand{\K}{\mathcal{K}}

\newcommand{\bx}{\mathbf{x}}
\newcommand{\by}{\mathbf{y}}

\title{Schoenberg Representations and Gramian Matrices of Mat\'ern Functions}

\author{Yong-Kum Cho
\and Dohie Kim
\and Kyungwon Park
\and Hera Yun
 }

\begin{document}

\maketitle

\bigskip

\begin{itemize}
\item[{}] {\bf Abstract.} We represent Mat\'ern functions in terms of Schoenberg's integrals
which ensure the positive definiteness and prove the systems of translates
of Mat\'ern functions form Riesz sequences in $L^2(\R^n)$ or Sobolev spaces.
Our approach is based on a new class of integral transforms that
generalize Fourier transforms for radial functions.
We also consider inverse multi-quadrics and obtain similar results.
\end{itemize}

\bigskip
{\small
\begin{itemize}
\item[{}]{\bf Keywords.} Bessel function, Fourier transform, Gramian matrix, Hankel-Schoenberg transform, inverse multi-quadrics, Mat\'ern function,
positive definite, Riesz sequence, Schoenberg matrix, Sobolev space.

\item[{}] 2010 Mathematics Subject Classification: 33C10, 41A05, 42B10, 60E10.
\end{itemize}}

\newpage

\section{Introduction}
In many areas of Mathematics, the functions of type
\begin{equation}
M_\alpha(z) = K_\alpha(z) z^\alpha\qquad(\alpha\in\R,\,z>0)
\end{equation}
arise frequently, referred to as the Mat\'ern functions, where
$K_\alpha(z)$ stands for the modified Bessel function of the second kind
of order $\alpha$.

Intimately connected is the family of functions of type
\begin{equation}
\phi_\beta(r) = (1+r^2)^{-\beta}\qquad(\beta>0, \,r\ge 0)
\end{equation}
whose radial extensions to the Euclidean spaces are referred to as the inverse multi-quadrics
in the theory of interpolations or spatial statistics.

In a fixed Euclidean space, both class of functions, if radially extended with suitably rearranged $\,\alpha, \beta,\,$
provide essential ingredients of Sobolev spaces.
In their pioneering work \cite{AK}, N. Aronszajn and K. T. Smith introduced the Sobolev space
$H^\alpha(\R^n),\,\alpha>0,\,$ as the space of Bessel potentials, that is, the convolutions
$\,(G_{\alpha/2}\ast u)(\bx),\,u\in L^2(\R^n),\,$ where $G_{\alpha/2}$ denotes the radial
extension of a special kind of Mat\'ern functions defined as follows.

\begin{definition} For a positive integer $n$ and $\,\alpha>0,$
\begin{equation}\label{G1}
G_\alpha(z) = \frac{1}{2^{\alpha-1 + \frac n2}\,\pi^{\frac n2}\,
\Gamma(\alpha)}\,K_{\alpha-\frac n2}(z) z^{\alpha - \frac n2}\qquad(z>0).
\end{equation}
For its radial extension to the Euclidean space $\R^n$, we write
$$G_\alpha(\bx) = G_\alpha(|\bx|), \quad |\bx| = \sqrt{\bx\cdot\bx}\qquad(\bx\in\R^n).$$
\end{definition}

A characteristic feature of the kernel $G_\alpha$ is the Fourier transform
\begin{align*}
\widehat{G_\alpha}(\xi) =\int_{\R^n} e^{-i \xi\cdot\bx}\,G_\alpha(\bx) d\bx= \left( 1+|\xi|^2\right)^{-\alpha}\,,
\end{align*}
which, together with the intrinsic properties of $K_{\alpha-n/2}$, enables the authors to obtain a comprehensive list of functional properties.
Let us state only a few of their list which are relevant to the present work (see also \cite{C}).

\begin{itemize}
\item[(a)] The Sobolev space $H^\alpha(\R^n)$ is identified with
\begin{equation*}
H^\alpha(\R^n) = \left\{ u\in L^2(\R^n) : \int_{\R^n} \left( 1+|\xi|^2\right)^{\alpha}|\widehat{u}(\xi)|^2 d\xi <\infty
\right\}.
\end{equation*}
In particular, $\,G_\beta\in H^\alpha(\R^n)\,$ if and only if $\,\beta>(2\alpha +n)/4.$
\item[(b)] $\,\left(G_\alpha\ast G_\beta\right)(\bx) = G_{\alpha +\beta}(\bx)\,$ for $\,\alpha>0,\,\beta>0.$
\item[(c)] In the case $\,\alpha>n/2,\,$ $G_{\alpha}$ is positive definite on $\R^n$.
The symmetric kernel $G_{\alpha}(\bx - \mathbf{y})$
is in fact a reproducing kernel for the Hilbert space $H^\alpha(\R^n)$ under the inner product
$$\big(u, v\big)_{H^\alpha(\R^n)} =
(2\pi)^{-n}\int_{\R^n} \widehat{u}(\xi)\,\overline{\widehat{v}(\xi)}\,(1+|\xi|^2)^\alpha\,d\xi.$$
\end{itemize}

Our primary purpose in the present work is to obtain a set of invariants for both classes
of functions, that is, those properties valid in any Euclidean space, related with
the positive definiteness and Fourier transforms.

We recall that a univariate function $\phi$ defined on the interval $[0, \infty)$
is said to be {\it positive semi-definite on $\R^n$}
if it satisfies
\begin{align}\label{G3}
\sum_{j=1}^N\sum_{k=1}^N \,\phi\left(\left|\bx_j - \bx_k\right|\right)\alpha_j \overline{\alpha_k}\,\ge\, 0
\end{align}
for any choice of $\,\alpha_1, \cdots, \alpha_N\in\mathbb{C}\,$ and
distinct points $\,\bx_1, \cdots, \bx_N\in\R^n,\,$ where $N$ is arbitrary. If equality in \eqref{G3} holds only if
$\,\alpha_1=\cdots=\alpha_N=0,\,$ then it is said to be {\it positive definite on $\R^n$}.

A univariate function which is positive semi-definite or positive definite on every $\R^n$
takes the following specific form:

\begin{itemize}
\item[{}]{\bf Criterion I} (I. J. Schoenberg \cite{Sc2}).
{\it A continuous function $\phi$ on $[0, \infty)$ is
positive semi-definite on every $\R^n$ if and only if
\begin{equation}\label{G4}
\phi(r) = \int_0^\infty e^{-r^2 t}\,d\nu(t)
\end{equation}
for a finite positive Borel measure $\nu$ on $[0, \infty).$
Moreover, if $\nu$ is not concentrated at zero,
then $\phi$ is positive definite on every $\R^n$. }
\end{itemize}

Due to the representation formula
\begin{equation}\label{G5}
 \phi_\beta(r) =  \frac{1}{\Gamma(\beta)}\int_0^\infty e^{-r^2 t}\,e^{-t} t^{\beta -1} dt\qquad(\beta>0),
\end{equation}
it is well known that each $\phi_\beta$ is positive definite on every $\R^n$ (see e.g. \cite{We}).

Our preliminary observation is the following.

\smallskip

\begin{theorem}\label{theorem1.1}
For $\,\alpha>0,\,$ we have
\begin{equation*}
\frac{2^{1-\alpha}}{\Gamma(\alpha)}\,K_\alpha(z) z^\alpha = \int_0^\infty e^{-z^2 t}\,f_\alpha(t) dt\qquad(z\ge 0),
\end{equation*}
where $f_\alpha$ denotes the probability density defined by
$$ f_\alpha(t) = \frac{1}{2^{2\alpha}\Gamma(\alpha)}\,\exp\left(-\frac{1}{4t}\right) t^{-\alpha-1}.$$
As a consequence, $M_\alpha$ is positive definite on every $\R^n$.
\end{theorem}

\smallskip

In order to find direct relationships between the functions $M_\alpha$ and $\phi_\beta$, without recourse to
their Euclidean extensions, we shall introduce a new class of integral transforms that incorporates Fourier transforms
for radial measures and Hankel transforms in certain sense.

\begin{definition}\label{def1}
For $\,\lambda>-1,\,$ let $J_\lambda$ denote the Bessel function of the first kind of order $\lambda$ and
define $\,\Omega_\lambda : \R\to \R\,$ by
\begin{align*}
\Omega_\lambda(t) &= \Gamma(\lambda+1)\left(\frac t2\right)^{-\lambda} J_\lambda(t)\\
&=\Gamma(\lambda+1)\sum_{k=0}^\infty\frac{(-1)^k}{k!\,\Gamma(\lambda +k +1)}\,\left(\frac t 2\right)^{2k}.
\end{align*}
\end{definition}

In the special case $\,\lambda = (n-2)/2\,,$ with $n$ a positive integer, $\Omega_\lambda$
arises on consideration of the Fourier transforms for radial functions on $\R^n$. To be specific,
if $F$ is integrable with $\,F(\bx)= f(|\bx|)\,$ for some univariate function $f$ on $[0, \infty)$,
then it is well known (see e.g. \cite{Stw}) that
\begin{align}\label{G6}
\widehat{F}(\xi) &= (2\pi)^{n/2} |\xi|^{-\frac{n-2}{2}}\int_0^\infty J_{\frac{n-2}{2}}
(|\xi|t) f(t) t^{n/2}dt\nonumber\\
&=\frac{2\pi^{n/2}}{\Gamma(n/2)} \int_0^\infty \Omega_{\frac{n-2}{2}}(|\xi|t) f(t) t^{n-1} dt.
\end{align}

More extensively, I. J. Schoenberg noticed that the Fourier transform of any radial measure on $\R^n$
is also representable in the above form and set up the following characterization (see also H. Wendland \cite{We}).

\begin{itemize}
\item[{}]{\bf Criterion II} (I. J. Schoenberg \cite{Sc1}, \cite{Sc2}).
{\it A continuous function $\phi$ on $[0, \infty)$ is
positive semi-definite on $\R^n$ if and only if
\begin{equation}\label{G7}
\phi(r) = \int_0^\infty \Omega_{\frac{n-2}{2}}(rt) d\nu(t)
\end{equation}
for a finite positive Borel measure $\nu$ on $[0, \infty)$.
Moreover, in the case when $\,d\nu(t) = f(t) t^{n-1} dt\,$ with continuous $f$,
$\phi$ is positive definite on $\R^n$ if and only if $\phi$ is nonnegative and non-vanishing.
}
\end{itemize}

Our generalization of Schoenberg's integrals or Fourier transforms for radial measures
takes the following form.

\begin{definition}
The Hankel-Schoenberg transform of order $\,\lambda>-1\,$
of a finite positive Borel measure $\nu$ on $[0, \infty)$ is defined by
\begin{equation*}
\phi(r) = \int_0^\infty\Omega_\lambda(rt)\,d\nu(t)\qquad(0\le r<\infty).
\end{equation*}
\end{definition}

For those Borel measures on $[0, \infty)$ which are absolutely
continuous with respect to Lebesgue measure, it is simple to express the Hankel-Schoenberg transforms
in terms of the classical Hankel transforms for which analogues of
the Fourier inversion theorem and Parseval's relations are available.

Our evaluations will be of the form
\begin{align}
\left(1+r^{2}\right)^{-\alpha-\lambda -1}  = c(\alpha, \lambda)\int_{0}^{\infty}\Omega_{\lambda}(rt)\big[K_{\alpha}(t)t^{\alpha}\big]t^{2\lambda+1}dt
\end{align}
for $\,\alpha+\lambda+1>0\,$ with an explicit positive constant $c(\alpha, \lambda)$.
By inversions and order-changing transforms, we shall obtain
a number of representation formulas for the Mat\'ern functions $M_\alpha$ in terms of $\phi_\beta$'s and
vice versa, which suits to Schoenberg's criterion and makes it possible to find
the Fourier transforms of their radial extensions to any Euclidean space.

In accordance with the notation of \cite{GMO}, we introduce

\begin{definition}
For a univariate function $\phi$ on $[0, \infty)$ and a set of distinct
points $\,X= \{\bx_j\}_{j\in\N}\subset\R^n,\,$
the Schoenberg matrix is defined to be
\begin{equation}
\mathbf{S}_X (\phi) = \Big[ \phi\big(\left|\bx_j - \bx_k\right|\big)\Big]_{j, \,k\in\mathbb{N}}.
\end{equation}
\end{definition}

The notion of Schoenberg matrix comes up instantly with an attempt to
construct an interpolating functional that matches the values of any function at each point of $X$.
To state briefly, if $\mathbf{S}_X (\phi)$ defines a bounded invertible operator on the space
$\ell^2(\N)$, then it is possible to construct a Lagrange-type radial basis sequence
$\,\left\{u_j^*\right\}_{j\in\N}\,$ by setting
$$u_j^*(\bx) = \sum_{k=1}^\infty c_{j, k}\,\phi(|\bx- \bx_k|),\quad j=1,2, \cdots, $$
and solving the infinite system $\,u_j^*(\bx_k) = \delta_{j, k},\,$ which has a unique solution
$\,\mathbf{c}_j = (c_{j, 1}, c_{j, 2},\cdots)\in\ell^2(\N)\,$ for each $j$. The functional
$$A_X(f)(\bx) =  \sum_{j=1}^\infty f(\bx_j)\,u_j^*(\bx),$$
definable on any class of functions, obviously interpolates $f$ at $X.$

The Schoenberg matrices arise under various guises in other fields of Mathematics.
In Functional Analysis, for example, it is common that $\mathbf{S}_X (\phi)$ coincides with
the Gramian matrix of a sequence obtained by translating another function $\psi$ by $X$
in an appropriate Hilbert space $H$, that is,
$$\mathbf{S}_X (\phi) = \Big[\big(\psi(|\cdot - \bx_j|),\,\psi(|\cdot - \bx_k|)\big)_H\Big]_{j, k\in\N}.$$
In such a circumstance, $\,\left\{\psi(|\bx- \bx_j|)\right\}_{j\in\N}\,$
is a Riesz sequence in $H$ if and only if $\mathbf{S}_X (\phi)$ defines a bounded invertible operator on $\ell^2(\N)$.

Our secondary purpose is to study the Schoenberg or Gramian matrices associated with
the Mat\'ern functions $M_\alpha$ as well as the functions $\phi_\beta$ with our focuses on their
boundedness and invertibility on $\ell^2(\N)$.

Our approaches are substantially based on the recent developments
\cite{GMO}, \cite{MS} of L. Golinskii {\it et al.} in which a list of
criteria for the boundedness and invertibility are established from several perspectives.
To illustrate, the authors devoted considerable portions of their work in studying
the $L^2$-based Gramian matrices associated to the Mat\'ern functions $M_\alpha$
and obtained their boundedness and invertibility on $\ell^2(\N)$ in the range $\,-n/4<\alpha\le 0.$

As we shall present below, we shall improve their results by extending the range to
$\,\alpha>-n/4\,$ and the boundedness and invertibility results to the aforementioned
Sobolev space-based Gramian matrices. In applications, we shall prove that the system of type
$\,\big\{M_\alpha(|\bx-\bx_j|)\big\}_{j\in\N},\,$
where $\,(\bx_j)\,$ is an arbitrary set of distinct points of $\R^n$, is a Riesz sequence
in $L^2(\R^n)$ or the Sobolev space of certain specified order.

In the same manner, the system of type $\,\big\{\phi_\beta(|\bx-\bx_j|)\big\}_{j\in\N}\,$ will be shown to be
a Riesz sequence in the Hilbert space of functions on $\R^n$ for which $\,\phi_\beta(|\bx-\by|)\,$ is a reproducing
kernel.

\section{Bessel functions $K_\alpha$}
In this section we collect some of the basic properties of $K_\alpha$ relevant to
the present work, most of which can be found in \cite{AS}, \cite{E} and \cite{Wa}.

For $\,\alpha\in\mathbb{C},\,$ the modified Bessel function $K_\alpha$ is defined by
\begin{align}
K_\alpha(z) &= \frac{\pi}{2}\left[\frac{\,I_{-\alpha}(z)-I_{\alpha}(z)\,}{\sin\left(\alpha\pi\right)}\right],\quad\text{where}\\
I_{\alpha}(z) &= \sum_{k=0}^{\infty}\frac{1}{k!\,\Gamma\left(k+\alpha+1\right)}\left(\frac{z}{2}\right)^{2k +\alpha}.\label{K1}
\end{align}
In the case when $\alpha$ happens to be an integer, $\,\alpha=n,\,$ this formula should be interpreted as
$\,K_n(z) = \lim_{\alpha\to\, n} K_\alpha(z).\,$ The Bessel functions $\,I_\alpha, \,K_\alpha\,$ form a fundamental
system of solutions to the differential equation
\begin{equation}\label{K2}
z^2\frac{d^2 u}{dz^2} + z \frac{du}{dz} - (z^2 + \alpha^2) u = 0.
\end{equation}
Hereafter, we shall be concerned only with $\,\alpha\in\R\,$ and $\,z>0.$

\begin{itemize}
\item[(K1)] By definition, it is evident $\,K_{-\alpha}(z) = K_\alpha(z).\,$ For each integer $n$,
a series expansion formula for $K_n(z)$ is also available. In particular,
\begin{equation}\label{K5}
K_0(z) = -\log (z/2) I_0(z) + \sum_{k=0}^\infty \frac{\psi(k+1)}{(k!)^2} \left(\frac z2\right)^{2k},
\end{equation}
where $\psi$ denotes the digamma function so that
$$\psi(1) = -\gamma, \quad \psi(k+1) = -\gamma + \sum_{j=1}^k\frac 1j\,,$$
with $\gamma$ being the Euler-Mascheroni constant.

\item[(K2)] For $\,\alpha>-1/2\,$ and $\,z>0,$ Schl\"afli's integrals state
\begin{align}\label{K3}
K_{\alpha}(z) &= \frac{\sqrt{\pi}}{\Gamma(\alpha + 1/2)}\,\left(\frac{z}{2}\right)^{\alpha}
\int_{1}^{\infty} e^{-zt}\left(t^{2}-1\right)^{\alpha-\frac{1}{2}}\,dt\nonumber\\
&=\sqrt{\frac{\pi}{2}\,} \frac{e^{-z} z^\alpha}{\Gamma(\alpha+1/2)}\,
\int_{0}^{\infty} e^{- zt} \left[ t\left( 1 + \frac t2\right)\right]^{\alpha - \frac 12}\,dt
\end{align}
in which the latter follows from the former by suitable substitutions.
Another form of Schl\"afli's integral reads
\begin{equation}\label{K4}
K_\alpha(z) = \frac 12 \int_{-\infty}^\infty e^{-z \cosh t - \alpha t}\, dt,
\end{equation}
which holds for any real $\alpha$ and $\,z>0.$ As a consequence, $K_\alpha(z)$ is
positive on the interval $(0, \infty)$.

\item[(K3)] From the differential equation \eqref{K2}, it follows plainly
$$\frac{d}{dz} \big[K_\alpha(z) z^\alpha\big] = - K_{\alpha -1}(z) z^\alpha.$$
By (K2), hence, the Mat\'ern function $\,M_\alpha(z) = K_\alpha(z) z^\alpha\,$ is positive and strictly decreasing
on the interval $(0, \infty)$.

\item[(K4)] Of great significance is the asymptotic behavior of $K_\alpha$ for $\,\alpha\ge 0.$
\begin{itemize}
\item[(i)] As $\,z\to 0,\,$ the series expansions \eqref{K1} and \eqref{K5} yield\footnote{To be more precise, \eqref{K1} shows
$$K_\alpha(z) = 2^{\alpha-1}\Gamma(\alpha) z^{-\alpha}\big[ 1 + O\left( z^{\alpha_*}\right)\big],$$
where $\,\alpha_* = \min (2\alpha, \,2)\,$ and \eqref{K5} shows
$$K_0(z) = -\log z + \log 2 -\gamma + \big[1-\log(z/2)\big] O\left(z^2\right).$$}
\begin{equation*}
K_\alpha(z) \,\sim\,\left\{\begin{aligned} &{2^{\alpha-1}\Gamma(\alpha) z^{-\alpha}} &{\quad\text{for}\quad \alpha>0},\\
&{-\log z} &{\quad\text{for}\quad \alpha =0}.\end{aligned}\right.
\end{equation*}
\item[(ii)] As $\,z\to\infty,$ a version of Hankel's asymptotic formula states
\begin{align*}
K_{\alpha}(z) = \sqrt{\frac{\pi}{2 z}\,}\,e^{-z}\left[ 1 + \frac{4\alpha^{2}-1}{8z} + O\left(\frac{1}{z^2}\right)\right].
\end{align*}
\end{itemize}

\item[(K5)] In the special case $\,\alpha = n + 1/2\,$ with $n$ an integer, it is simple to express
$K_\alpha$, and hence the Mat\'ern function $M_\alpha$, in closed forms on evaluation of Schl\"afli's integral \eqref{K3}.
To state $M_\alpha$ explicitly,
\begin{align}\label{K6}
M_{n+ \frac 12}(z) &= \sqrt{\frac{\pi}{2}\,}\,e^{-z} z^{n}\sum_{k=0}^n \frac{(n+k)!}{k! (n-k)!}\,(2z)^{-k},\nonumber\\
M_{-n-\frac 12} (z) &= \sqrt{\frac{\pi}{2}\,}\,e^{-z} z^{-n-1}\sum_{k=0}^n \frac{(n+k)!}{k! (n-k)!}\,(2z)^{-k},
\end{align}
where $n$ is a nonnegative integer. A list of positive orders reads
\begin{align}
M_{\frac 12}(z)&=  \sqrt{\frac{\pi}{2}\,}\, e^{-z}\,,\nonumber\\
M_{\frac 32}(z)&=  \sqrt{\frac{\pi}{2}\,}\,(1+z)e^{-z}\,,\nonumber\\
M_{\frac 52}(z)&= \sqrt{\frac{\pi}{2}\,}\,\left(3 + 3z+ z^{2}\right)e^{-z}
\end{align}
which are of considerable interest in spatial statistics (see \cite{G1}, \cite{G2}).
A list of negative orders reads
\begin{align}
M_{-\frac 12}(z)&=  \sqrt{\frac{\pi}{2}\,} \,\frac{e^{-z}}{z}\,,\nonumber\\
M_{-\frac 32}(z)&=  \sqrt{\frac{\pi}{2}\,}\,\left(\frac{1}{z^2} + \frac{1}{z^3}\right)e^{-z}\,,\nonumber\\
M_{-\frac 52}(z)&= \sqrt{\frac{\pi}{2}\,}\,\left(\frac{1}{z^3} + \frac{3}{z^4} +\frac{3}{z^5}\right)e^{-z}\,.
\end{align}
\end{itemize}

\section{Hankel-Schoenberg transforms}
The purpose of this section is to establish basic properties of the
Hankel-Schoenberg transforms which will be used subsequently.

To begin with, we list the following properties on the kernels $\Omega_\lambda$
which are deducible from those on the Bessel functions $J_\lambda$
(\cite{E}, \cite{Wa}).

\begin{itemize}
\item[(J1)] Each $\Omega_\lambda$ is of class $C^\infty(\R)$, even and uniformly bounded by $\,1=\Omega_\lambda(0).\,$
A theorem of Bessel-Lommel states that it is an oscillatory function with an infinity of positive simple zeros.
A modification of Hankel's asymptotic formula for $J_\lambda$ shows that as $\,t\to\infty,$
\begin{equation*}
\Omega_\lambda(t) = \frac{\Gamma(\lambda+1)}{\sqrt{\pi}} \left(\frac t2\right)^{-\lambda -1/2}
\left[\cos\left(t - \frac{(2\lambda+1)\pi}{4}\right) + O\left( t^{-1}\right)\right].
\end{equation*}

\item[(J2)] For $\,\lambda>-1/2,\,$ Poisson's integral reads
\begin{align*}
\Omega_\lambda(t) = \frac{2}{B\left(\lambda + 1/2\,,\,1/2\right)}\,
\int_0^1 \cos(t s)\, (1-s^2)^{\lambda -\frac 12}\,dt,
\end{align*}
where $B$ stands for the Euler beta function defined by
$$B(a, \,b) = \int_0^1 t^{a-1} (1-t)^{b-1} dt\qquad(a>0, \,b>0).$$

\item[(J3)] By Liouville's theorem, $\Omega_\lambda$ is expressible
in finite terms by algebraic and trigonometric functions
if and only if $2\lambda$ is an odd integer. Indeed,
the Lommel-type recurrence formula
\begin{align*}
\Omega_\lambda(t) -\Omega_{\lambda-1}(t)
 = \frac{t^2}{\,4\lambda(\lambda+1)\,}\,\Omega_{\lambda+2}(t)
\qquad(\lambda>-1)
\end{align*}
may be used to evaluate $\Omega_{n + 1/2}$ for any integer $n$ together with
$$\Omega_{-\frac 12}(t) = \cos t,\quad \Omega_{\frac 12}(t) = \frac{\sin t}{t}\,.$$
\end{itemize}

The Hankel transforms of a function $f$ refer to the integrals
$$\int_0^\infty J_\lambda(rt) f(t) t dt\qquad(\lambda\in\mathbb{C}).$$
It follows by definition that the Hankel-Schoenberg transforms
can be written in terms of the Hankel transforms whenever $\nu$ admits an integrable density $f$,
that is, $\,d\nu(t) = f(t) dt.\,$
The Hankel-Watson inversion theorem (\cite{Wa}) states that if $\,\lambda\ge -1/2\,$
and $\,f(t)\sqrt t\,$ is integrable on $[0, \infty)$, then
\begin{equation*}
\int_0^\infty J_\lambda(rt)\left[\int_0^\infty J_\lambda(ru) f(u)u du\right] rdr
= \frac{\,f(t+0) + f(t-0)\,}{2}
\end{equation*}
at every $\,t>0\,$ such that $f$ is of bounded variation in a neighborhood of $t$.

An obvious modification yields the following inversion formula which may serve as an alternative of
the Fourier inversion theorem for radial functions.

\begin{theorem}\label{inversion} {\rm (Inversion)}
For $\,\lambda\ge -1/2,\,$ assume that
\begin{equation}\label{invc1}
\int_0^\infty |f(t)| t^{-\lambda-1/2}\,dt <\infty.
\end{equation}

Then the following holds for every $\,t>0\,$ at which $f$ is continuous:
\begin{align*}
\left\{\aligned &{\phi(r) = \int_0^\infty \Omega_\lambda(rt) f(t) dt\quad \text{implies}}\\
 &{f(t) = \frac{t^{2\lambda+1}}{4^\lambda\left[\Gamma(\lambda+1)\right]^2}\,
 \int_0^\infty \Omega_\lambda(rt)\, \phi(r) r^{2\lambda+1}dr.}\endaligned\right.
\end{align*}
\end{theorem}

A version of Parseval's theorem is deducible from its equivalent for the Hankel transforms
in a trivial manner.

\begin{theorem}\label{Parseval}
{\rm (Parseval's relation)}
For $\,\lambda>-1\,,$ let
\begin{align*}
\phi_j(r) = \int_0^\infty\Omega_\lambda(rt) f_j(t) t^{2\lambda +1} dt,\quad j=1, 2.
\end{align*}
If both integrals are absolutely convergent, then
\begin{equation*}
\int_0^\infty f_1(t) f_2(t) t^{2\lambda+1} dt = \frac{1}{4^\lambda\left[\Gamma(\lambda+1)\right]^2}\,
\int_0^\infty \phi_1(r) \phi_2(r) r^{2\lambda+1} dr.
\end{equation*}
\end{theorem}

\begin{lemma}\label{basic}
For $\,\lambda>\rho>-1\,$ and $\,r\ge 0,$
\begin{equation*}
\Omega_\lambda(r) = \frac{2}{B(\rho +1, \,\lambda-\rho)} \int_0^\infty\Omega_\rho(rt)
(1-t^2)_+^{\lambda-\rho-1}t^{2\rho+1} dt.
\footnote{As usual, we write $\,x_+ = \max\,(x, 0)\,$ for $\,x\in\R.\,$}\end{equation*}
\end{lemma}

\begin{proof} If $\nu$ denotes the probability measure
$$d\nu(t) = \frac{2}{B(\rho +1, \,\lambda-\rho)} \,(1-t^2)_+^{\lambda-\rho-1}t^{2\rho+1} dt,$$
then it has finite moments of all orders with
$$\int_0^\infty t^{2k} d\nu(t) = \frac{\Gamma(k+\rho+1)}{\Gamma(\rho+1)}\cdot
\frac{\Gamma(\lambda+1)}{\Gamma(k+\lambda+1)},\quad k=0,1,\cdots.$$
It follows from integrating termwise, readily justified, that
\begin{align*}
\int_0^\infty\Omega_\rho(rt) d\nu(t) &= \Gamma(\rho+1)
\sum_{k=0}^\infty\frac{(-1)^k}{k!\,\Gamma(k+\rho+1)}\left(\frac r2\right)^{2k}
\int_0^\infty t^{2k} d\nu(t)\\
&= \Gamma(\lambda+1)\sum_{k=0}^\infty\frac{(-1)^k}{k!\,\Gamma(k+\lambda +1)}\left(\frac r2\right)^{2k}\\
&=\Omega_\lambda(r).
\end{align*}
\end{proof}

The Hankel-Schoenberg transforms may be regarded as a generalization of the radial Fourier transforms
or Schoenberg's integrals due to the following order-changing interrelations.

\begin{theorem}\label{orderwalk}
Let $\,\lambda>\frac{n-2}{2}\,$ with $n$ a positive integer.
For any finite positive Borel measure $\nu$ on $[0, \infty)$ which is not concentrated at zero,
its Hankel-Schoenberg transform of order $\lambda$ can be represented as
\begin{align*}
& \int_0^\infty \,\Omega_{\lambda}(rt) d\nu(t) =
\int_0^\infty \,\Omega_{\frac{n-2}{2}}(rt)\, W_\lambda(\nu)(t) t^{n-1} dt,\quad\text{where}\\
& W_\lambda(\nu)(t) = \frac{2}{B\left(\frac n2, \,\lambda +1 -\frac n2\right)}
\int_0^\infty\left( 1- \frac{t^2}{s^2}\right)^{\lambda - \frac n2}_+ s^{-n} d\nu(s).
\end{align*}

\smallskip

\noindent
Moreover, $\,d\mu(t) = W_\lambda(\nu)(t) t^{n-1} dt\,$ defines
a finite positive Borel measure on $[0, \infty)$ with the total mass $\mu\left([0, \infty)\right) =
\nu\left([0, \infty)\right).$
\end{theorem}

\begin{proof}
As a special case of Lemma \ref{basic}, the choice $\,\rho=\frac{n-2}{2}\,$ gives
\begin{equation}\label{O1}
\Omega_\lambda(r) = \frac{2}{B\left(\frac n2, \,\lambda+1-\frac n2\right)} \int_0^\infty\Omega_{\frac{n-2}{2}}(rs)
(1-s^2)_+^{\lambda-\frac n2}s^{n-1} ds,
\end{equation}
whence the result follows by interchanging the order of integrations.

Since
\begin{align*}
\int_0^\infty\left( 1- \frac{t^2}{s^2}\right)^{\lambda - \frac n2}_+ t^{n-1} dt
= \frac{s^n}{2}\int_0^1 (1-u)^{\lambda-\frac n2} u^{\frac n2 -1} du
\end{align*}
for each $\,s>0,$ it is straightforward to find
\begin{align*}
\mu([0, \infty)) &= \int_0^\infty W_\lambda(\nu)(t) t^{n-1} dt\\
&= \frac{2}{B\left(\frac n2, \,\lambda +1 -\frac n2\right)}
\int_0^\infty\int_0^\infty\left( 1- \frac{t^2}{s^2}\right)^{\lambda - \frac n2}_+ s^{-n} d\nu(s) t^{n-1} dt\\
&= \frac{2}{B\left(\frac n2, \,\lambda +1 -\frac n2\right)}
\int_0^\infty \int_0^\infty\left( 1- \frac{t^2}{s^2}\right)^{\lambda - \frac n2}_+ t^{n-1} dt s^{-n} d\nu(s)\\
&=\nu([0, \infty)).
\end{align*}

\end{proof}

\begin{remark}
A positive Borel measure $\nu$ on $[0, \infty)$ is concentrated at zero
if it is a constant multiple of Dirac mass at zero, that is,
$\,\nu = c\,\delta_0\,$ with $\,c>0.$ For such a Borel measure $\nu$,
its Hankel-Schoenberg transform is simply
$$ \int_0^\infty\Omega_{\lambda}(rt) d\nu(t) = c\,\Omega_\lambda(0) = c.$$
\end{remark}

\section{Schoenberg representations}
Our aim in this section is to set up Schoenberg's representations for
Mat\'ern functions which ensure their positive definiteness.

\begin{lemma}\label{lemmaS0}
For $\,\alpha\in\R\,$ and $\,z>0,\,$ we have
\begin{equation}\label{S1}
K_\alpha(z) z^\alpha = 2^{-\alpha -1}\int_0^\infty
\exp\left(-z^2 t - \frac {1}{4t}\right) t^{-\alpha -1} dt.
\end{equation}
\end{lemma}

\begin{proof} For any real $\alpha$ and $\,z>0,$ if we make substitution
$\,z e^{-t} = 2s\,$ in the second form of Schl\"afli's integral \eqref{K4}, then
\begin{align*}
K_\alpha(z) &= \frac 12\int_{-\infty}^\infty \exp\left(-z\cosh t-\alpha t\right) dt\\
&= 2^{\alpha-1}z^{-\alpha}\int_0^\infty \exp\left(-s - \frac{z^2}{4s}\right) s^{\alpha -1} ds
\end{align*}
from which \eqref{S1} follows on making another substitution $\,s= 1/4t.$
\end{proof}

In the case $\,\alpha>0,\,$ it follows from the asymptotic behavior of $K_\alpha$
near zero, as stated in (K4), that the Mat\'ern function $M_\alpha$ is well defined
as a continuous function on $[0, \infty)$ with the limiting value $\,M_\alpha(0) = 2^{\alpha -1}\,\Gamma(\alpha).$
For this reason, it will be convenient to consider the following types of Mat\'ern functions
which are frequently used in many fields (see e.g. \cite{G1}).

\begin{definition}
For $\,\alpha>0,\,$ put
\begin{equation}\label{M}
\M_\alpha(z) = \frac{2^{1-\alpha}}{\Gamma(\alpha)}\,K_\alpha(z) z^\alpha\qquad(z\ge 0).
\end{equation}
\end{definition}

We recall that a function $\phi$ is said to be {\it completely
continuous on $[0, \infty)$} if it is continuous on $[0, \infty)$ and satisfies
the condition $\,(-1)^m\phi^{(m)}(z)\ge 0\,$ for all nonnegative integers $m$ and $\,z>0\,$ (see e.g. \cite{We}).

\begin{theorem}\label{corollaryS1}{\rm (Theorem \ref{theorem1.1})}
For $\,\alpha>0,\,$ we have
\begin{equation*}
\M_\alpha(z) = \int_0^\infty e^{-z^2 t} f_\alpha(t) dt\quad\,\,(z\ge0),
\end{equation*}
where $f_\alpha$ is the continuous probability density on $[0, \infty)$ defined by
\begin{equation*}
f_\alpha(t) = \left\{\begin{aligned} &{\frac{1}{4^\alpha\Gamma(\alpha)}\,
\exp\left(-\frac{1}{4t}\right) t^{-\alpha-1}} &{\text{for}\quad t>0},\\
&{\qquad\qquad\,\, 0} &{\text{for}\quad t=0}.\end{aligned}\right.
\end{equation*}
As a consequence, $\M_\alpha$  is continuous and positive definite on every $\R^n$.
In addition, the function $\M_\alpha\left(\sqrt{z}\,\right)$ is also positive definite on every $\R^n$ and
completely continuous on $[0, \infty)$.
\end{theorem}

\begin{proof}
As it is elementary to verify that $f_\alpha$ is a continuous probability density on $[0, \infty)$,
the statements on $\M_\alpha(z)$ are immediate consequences of Lemma \ref{lemmaS0} and Schoenberg's
Criterion I on the positive definiteness.

In the special case $\,\alpha=1/2,\,$ we have
\begin{equation}\label{S2}
e^{-z} = \int_0^\infty e^{-z^2 t} f_{1/2}(t) dt\quad\,\,(z\ge0),
\end{equation}
whence it is straightforward to deduce the integral representations
\begin{align}
\M_\alpha\left(\sqrt z\,\right) &= \int_0^\infty e^{-z t} f_\alpha(t) dt\label{S3}\\
&=\int_0^\infty e^{-z^2 u} g_\alpha(u) du\label{S4},
\end{align}
where $g_\alpha$ stands for the function defined by $\,g_\alpha(0) =0\,$ and
\begin{equation*}
g_\alpha(u) = \frac{u^{-3/2}}{2^{2\alpha+1}\sqrt\pi\,\Gamma(\alpha)}\int_0^\infty
\exp\left(-\frac{1}{4t} -\frac{t^2}{4u}\right) t^{-\alpha} dt
\end{equation*}
for $\,u>0.\,$ As readily verified, $g_\alpha$ is
a continuous probability density on $[0, \infty)$ and hence it follows from \eqref{S4}
and Schoenberg's Criterion I that the function
$\M_\alpha\left(\sqrt{z}\,\right)$ is positive definite on every $\R^n$.

That it is completely continuous on $[0, \infty)$ is a consequence of \eqref{S3}.\footnote{It may be proved either by differentiating
under the integral sign or by applying the well-known theorem of Bernstein-Hausdorff-Widder which states that
a function $f$ is completely
continuous on $[0, \infty)$ if and only if
$$ f(r) = \int_0^\infty e^{-rt} d\mu(t)\qquad(r\ge 0)$$
for some finite positive Borel measure $\mu$ on $[0, \infty)$ (see e.g. \cite{We}).}
\end{proof}

We are now concerned with the second form of Schoenberg's integrals.
For the sake of computational facilitation as well as inversion, it is advantageous to consider
the Hankel-Schoenberg transforms.

As it is conventional, we shall use the notation of Pochhammer and Barnes for the generalized hypergeometric functions
\begin{equation*}
{}_pF_q\left(a_1, \cdots, a_p;\,b_1, \cdots, b_q;\,x\right)
=\sum_{k=0}^\infty\frac{\left(a_1\right)_k\cdots\left(a_p\right)_k}{k!
\left(b_1\right)_k\cdots\left(b_q\right)_k}\,x^k
\end{equation*}
in which the symbol $(a)_k$ for a non-zero real number $a$ stands for
\begin{equation*}
(a)_k = \left\{\begin{aligned} &{a(a+1)\cdots (a+k-1)} &{\text{for}
\quad k\ge 1}, \\
&{\qquad 1} &{\text{for} \quad k = 1}.\end{aligned}\right.
\end{equation*}

The following is easily obtainable from Scl\"afli's integrals.
As it is known, however, we shall omit the proof (see \cite{AS}, \cite{E}, \cite{Wa}).

\begin{lemma}\label{lemmaS1}
For $\,\alpha\in\R\,$ and $\,\beta>|\alpha|,\,$ we have
\begin{equation}\label{S5}
\int_0^\infty K_{\alpha}(t)t^{\beta-1}dt = 2^{\beta-2}
\Gamma\left(\frac{\beta+\alpha}{2}\right)\Gamma\left(\frac{\beta-\alpha}{2}\right).
\end{equation}
\end{lemma}

\begin{lemma}\label{lemmaS2}
Let $\,\alpha\in\R\,$ and $\,\beta>|\alpha|.\,$ For the probability measure
$$d\nu(t) =\frac{1}{2^{\beta-2}
\Gamma\left(\frac{\beta+\alpha}{2}\right)\Gamma\left(\frac{\beta-\alpha}{2}\right)}\,
K_{\alpha}(t)t^{\beta-1}dt,$$
the Hankel-Schoenberg transform of order $\lambda>-1$ is given by
\begin{equation}\label{S6}
\int_{0}^{\infty}\Omega_{\lambda}(rt)d\nu(t)
={}_2F_{1}\left(\frac{\beta-\alpha}{2},\,\frac{\beta+\alpha}{2};\,\lambda+1;\,-r^2\right)\,.
\end{equation}
\end{lemma}

\begin{proof}
A simple modification of \eqref{S5} yields
 \begin{align*}
 \int_0^\infty t^{2k}d\nu(t)=2^{2k} \left(\frac{\beta+\alpha}{2}\right)_{k}
 \left(\frac{\beta-\alpha}{2}\right)_{k},\quad k=0,1,2,\cdots.
 \end{align*}
 Integrating termwise, we deduce
 \begin{align*}
 \int_{0}^{\infty}\Omega_{\lambda}(rt)d\nu(t)&=
 \sum_{k=0}^\infty \frac{\left(-1\right)^k}{k!\left(\lambda+1\right)_k} \left(\frac{r}{2}\right)^{2k}
 \int_0^\infty t^{2k} dt\\
 &=\sum_{k=0}^\infty
  \frac{\left(\frac{\beta+\alpha}{2}\right)_k\left(\frac{\beta-\alpha}{2}\right)_k}{k!\left(\lambda+1\right)_k}
  \left(-r^2\right)^k,
  \end{align*}
  which is equivalent to the stated formula \eqref{S6}.
\end{proof}

By obvious cancellation effects, the generalized hypergeometric function \eqref{S6}
reduces to the binomial series expansion in the case $\,\beta=\alpha + 2(\lambda+1)\,$
or $\,\beta=-\alpha + 2(\lambda+1).\,$ To be precise, we have the following
general results which include Schoenberg's representations for Mat\'ern functions.

\begin{theorem}\label{theoremS1} Let $\,\lambda>-1\,$ and $\,\alpha+\lambda+1>0.$ For each $\,r\ge 0,$ we have
\begin{align}\label{S7}
(1+r^{2})^{-\alpha-\lambda-1} &=\frac{1}{2^{\alpha+2\lambda}\Gamma(\lambda+1)\Gamma(\alpha+\lambda+1)}
\nonumber\\
&\qquad\times\quad\int_{0}^{\infty}\Omega_{\lambda}(rt) \big[K_{\alpha}(t) t^{\alpha}\big] t^{2\lambda+1} dt.
\end{align}
Moreover, if $\, 2\alpha +\lambda +3/2>0\,$ in addition, then for each $\,z>0,$
\begin{align}\label{S8}
K_\alpha(z) z^\alpha =\frac{2^\alpha\Gamma(\alpha+\lambda+1)}{\Gamma(\lambda+1)}
\int_0^\infty \Omega_\lambda(zt) (1+ t^2)^{-\alpha-\lambda-1} t^{2\lambda+1} dt.
\end{align}
\end{theorem}

\begin{proof}
Formula \eqref{S7} follows from the special case $\,\beta=\alpha+2\lambda+2\,$ of \eqref{S6},
Lemma \ref{lemmaS2}, and Newton's binomial theorem
 \begin{align*}
  \sum_{k=0}^{\infty}\frac{(\alpha+\lambda+1)_{k}}{k!}\,(-r^2)^{k} =(1+r^2)^{-\alpha-\lambda-1}.
  \end{align*}

As the function $\,f(t) = K_{\alpha}(t) t^{\alpha +2\lambda+1} \,$ is continuous on $(0, \infty)$ and
\begin{align*}
\int_0^\infty |f(t)| t^{-\lambda-1/2} dt &= \int_0^\infty K_{\alpha}(t) t^{\alpha +\lambda+1/2} dt\\
&=  2^{\alpha + \lambda -1/2}
\Gamma\left(\alpha + \frac{2\lambda +3}{4}\right)\Gamma\left(\frac{2\lambda+3}{4}\right)<\infty
\end{align*}
by Lemma \ref{lemmaS1}, applicable due to the condition $\,2\alpha +\lambda +3/2>0,\,$
\eqref{S8} follows from inverting \eqref{S7} in accordance with Theorem \ref{inversion}.
\end{proof}

Choosing $\,\alpha, \lambda\,$ suitably or regarding them as variable parameters, one may exploit these formulas
from several perspectives. If we are concerned with the Fourier transforms in a fixed
Euclidean space $\R^n$, for example, the first formula may be applied to yield the following.

\begin{itemize}
\item[(a)] For $\,\alpha>0,\,$  if we recall \eqref{G1}
\begin{equation*}
G_\alpha(z) = \frac{1}{2^{\alpha-1 + \frac n2}\,\pi^{\frac n2}\,
\Gamma(\alpha)}\,K_{\alpha-\frac n2}(z) z^{\alpha - \frac n2},
\end{equation*}
the special case $\,\lambda = (n-2)/2\,$ of \eqref{S7} yields
\begin{equation*}
(1+r^2)^{-\alpha} = \frac{2\pi^{n/2}}{\Gamma\left(n/2\right)} \int_0^\infty\Omega_{\frac{n-2}{2}} (rt) G_\alpha(t) t^{n-1} dt
\end{equation*}
so that we obtain the Fourier transform formula
\begin{equation}\label{S9}
\widehat{G_\alpha}(\xi) = (1 +|\xi|^2)^{-\alpha}.
\end{equation}

\item[(b)] As $\alpha$ varies over $\,\alpha> 0,\,$ \eqref{S9} expresses the inverse multi-quadrics
of any positive order in terms of the Fourier transforms of $G_\alpha(\bx)$. On the contrary, Hankel-Schoenberg
transform formula \eqref{S7} enables us to obtain such Fourier representations by varying $\lambda$ with a fixed $\alpha$.

To be specific, let us fix $\,\alpha>-n/2\,$ and set
\begin{align}
F_{\alpha, \lambda}(z) &= \frac{1}{2^{\alpha + 2\lambda} \pi^{\frac n2}\Gamma\left(\lambda +1 -\frac n2\right)\Gamma(\alpha + \lambda +1)}
\nonumber\\
&\qquad\qquad\times\quad \int_z^\infty (s^2 - z^2)^{\lambda -\frac n2} \big[K_\alpha(s) s^\alpha\big] s ds
\end{align}
for $\,\lambda>(n-2)/2.$ By Theorem \ref{orderwalk}, we may put \eqref{S7} in the form
\begin{equation*}
(1+r^2)^{-\alpha - \lambda-1} = \frac{2\pi^{n/2}}{\Gamma\left(n/2\right)} \int_0^\infty\Omega_{\frac{n-2}{2}} (rt) F_{\alpha, \lambda}(t) t^{n-1} dt.
\end{equation*}
If we write $\, F_{\alpha, \lambda}(\bx) =  F_{\alpha, \lambda}(|\bx|),\,\bx\in\R^n,\,$ then
\begin{equation}\label{S10}
\widehat{ F_{\alpha, \lambda}}(\xi) = (1 +|\xi|^2)^{-\alpha -\lambda-1}.
\end{equation}

As $\lambda$ varies in the range $\,\lambda>(n-2)/2,$ this Fourier transform formula represents
the inverse multi-quadrics of order greater than $\,\alpha + n/2.$
\end{itemize}

A noteworthy feature of Mat\'ern functions is the following invariance
which follows immediately from \eqref{S8} by reformulation.

\begin{corollary}\label{corollaryS1}
If $\,\alpha>0,\,$ then for any $\,\lambda>-1,$
\begin{align}\label{S11}
\M_\alpha(z)  = \int_0^{\infty}\Omega_\lambda (zt)\, d\nu_{\alpha, \lambda}(t)
\qquad(z\ge 0),
\end{align}
where $\nu_{\alpha, \lambda}$ denotes the probability measure on $[0, \infty)$ defined by
\begin{align*}
d\nu_{\alpha, \lambda}(t)= \frac{2}{B(\alpha, \,\lambda+1)}(1+t^2)^{-\alpha-\lambda-1}
t^{2\lambda+1} dt.
\end{align*}
\end{corollary}

\begin{remark}
In view of Schoenberg's Criterion II, this integral formula with $\,\lambda = (n-2)/2\,$
provides another proof of the positive definiteness of the Mat\'ern functions.
In particular, the choice of $\,n=1\,$ gives
\begin{align*}
\M_\alpha(z) = \frac{2}{B\left(\alpha,\,1/2\right)}
\int_{0}^{\infty}\frac{\cos (zt)\,dt}{\,\left(1 + t^{2}\right)^{\alpha+ 1/2}\,},
\end{align*}
the formula obtained by Basset, Malmst\'en and Poisson (see \cite{Wa}).
\end{remark}

\section{Schoenberg matrices on $\ell^2(\N)$}
In this section we shall investigate whether Schoenberg matrices of Mat\'ern
functions or inverse multi-quadrics, in a fixed Euclidean space $\R^n$,
give rise to bounded invertible operators on the Hilbert space
$\ell^{2}(\mathbb{N})$.

As it is common in the theory of scattered data approximations, we shall deal with arbitrary
sets of type $\,X = \left\{ \bx_{j}\in \R^n : j\in\mathbb{N}\right\}\,$ satisfying
\begin{equation}\label{M1}
\delta(X) = \inf_{j\neq k} \left|\bx_{j}-\bx_{k}\right|>0,\quad \dim\left[ {\rm span}(X)\right] = d
\end{equation}
for some $\,1\le d\le n.$
Our analysis will be based on the following.

\begin{proposition}\label{propM} {\rm (\cite{GMO})}
Let $f$ be a nonnegative function defined on $[0, \infty)$.

\begin{itemize}
\item[\rm(i)] Suppose $f$ is monotone decreasing, $\,f(0) =1\,$ and
$\,f(t) t^{d-1}\,$ is integrable on $[0, \infty).$
Then the Schoenberg matrix $\,\mathbf{S}_{X}(f)$ defines a bounded self-adjoint operator on $\ell^{2}(\mathbb{N})$ with
\begin{equation*}
\left\|\mathbf{S}_{X}(f)\right\|\le 1+\frac{ d(5^d-1)}{[\delta(X)]^d}\int_{0}^{\infty}f(t)t^{d-1}dt.
\end{equation*}
Moreover, if $X$ satisfies the additional separation assumption
\begin{equation*}
\delta(X) > \left[d(5^d-1)\int_{0}^{\infty}f(t)t^{d-1}dt\right]^{1/d},
\end{equation*}
then $\,\mathbf{S}_{X}(f)$ defines a bounded invertible operator on $\ell^{2}(\mathbb{N})$.

\item[\rm(ii)] Suppose $\,n\ge 2\,$ and $f$ admits an integral representation
$$f(r) = \int_0^\infty e^{-r^2 t}\, d\nu(t)\quad(r\ge 0)$$
for a finite positive Borel measure $\nu$ such that it is equivalent to Lebesgue measure on $[0, \infty)$
and satisfies the moment condition
$$\int_0^\infty t^{-d/2}\, d\nu(t)<\infty.$$
Then $\,\mathbf{S}_{X}(f)$ defines a bounded invertible operator on $\ell^{2}(\mathbb{N})$.
\end{itemize}
\end{proposition}

\begin{remark}
A positive Borel measure $\nu$ on $[0, \infty)$ is equivalent to Lebesgue measure $|\cdot|$ if both are absolutely
continuous with respect to each other. By the Radon-Nikodym theorem, a necessary and sufficient condition
for $\nu$ to be equivalent to Lebesgue measure is that $\,d\nu(t) = p(t) dt\,$ for a nonnegative density $p$
such that $\,{\rm supp} (p) = [0, \infty)\,$ and
$$\int_I p(t) dt =0 \,\Longleftrightarrow\, |I|=0$$
for any Borel set $\,I\subset [0, \infty).$

\end{remark}

As the operator norm bound and the invertibility condition of part (i) are slightly different from
the original ones presented in \cite{GMO}, we shall give a review of their proof for part (i) in the appendix.

Now that Schoenberg's representations are available for Mat\'ern functions of type \eqref{M},
it is a simple matter to prove the following.

\begin{theorem}\label{theoremM1}
Let $X$ be an arbitrary set of points of $\R^n$ satisfying \eqref{M1}. For $\,\alpha>0,\,$
consider the Schoenberg matrix of $\M_\alpha$,
$$\mathbf{S}_X\left(\M_\alpha\right) = \Big[ \M_\alpha\left(\bx_j - \bx_k\right)\Big]_{j, \,k\in\mathbb{N}}.$$

\begin{itemize}
\item[\rm(i)] $\,\mathbf{S}_X\left(\M_\alpha\right)$ defines a bounded self-adjoint operator on $\ell^{2}(\mathbb{N})$ with
$$\left\| \mathbf{S}_X\left(\M_\alpha\right)\right\| \le 1 + \frac{d\, 2^{d-1} (5^d -1) \Gamma\left(\alpha + \frac d2\right)\Gamma\left(\frac d2\right)}
{\left[\delta(X)\right]^d\,\Gamma(\alpha)}\,.$$

\item[\rm(ii)] For $\,n\ge 2,\,$ $\,\mathbf{S}_X\left(\M_\alpha\right)$ defines a bounded invertible operator on $\ell^{2}(\mathbb{N})$.
In the case $\,n=d=1,\,$ if $X$ satisfies the additional assumption
$$\delta(X) > \frac{4\,\Gamma\left(\alpha + \frac 12\right)\Gamma\left(\frac 12\right)}
{\Gamma(\alpha)}\,,$$
then it defines a bounded invertible operator on $\ell^{2}(\mathbb{N})$.
\end{itemize}
\end{theorem}

\begin{proof}
An application of Lemma \ref{lemmaS1} gives
\begin{align*}
\int_0^\infty \M_\alpha(t) t^{d-1} dt &= \frac{2^{1-\alpha}}{\Gamma(\alpha)}\int_0^\infty K_\alpha(t) t^{\alpha + d-1} dt\\
&= \frac{2^{d-1}\Gamma\left(\alpha + \frac d2\right)\Gamma\left(\frac d2\right)}
{\Gamma(\alpha)}\,.
\end{align*}

Since $\,\M_\alpha(0) = 1\,$ and $\M_\alpha$ is strictly decreasing on the interval $[0, \infty)$ as it is noted in (K3),
the criterion in the first part of Proposition \ref{propM} is applicable and part (i) follows with the stated operator norm bound.

Concerning part (ii), we invoke Corollary \ref{corollaryS1} to represent
$$\M_\alpha(z) = \int_0^\infty e^{-z^2 t} f_\alpha(t) dt \qquad(z\ge 0)$$
in which $f_\alpha$ stands for the probability density
\begin{equation*}
f_\alpha(t) = \left\{\begin{aligned} &{\frac{1}{4^\alpha\Gamma(\alpha)}\,
\exp\left(-\frac{1}{4t}\right) t^{-\alpha-1}} &{\text{for}\quad t>0},\\
&{\qquad\qquad\,\, 0} &{\text{for}\quad t=0}.\end{aligned}\right.
\end{equation*}

Since the measure determined by $\,f_\alpha(t) dt\,$ is obviously equivalent to Lebesgue measure
on $[0, \infty)$ and it is elementary to compute
$$\int_0^\infty t^{-d/2} f_\alpha(t) dt = \frac{2^d\Gamma\left(\alpha + \frac d2\right)}{\Gamma(\alpha)} <\infty,$$
the criterion in the second part of Proposition \ref{propM} implies the invertibility of  $\,\mathbf{S}_X\left(\M_\alpha\right)$
in the case $\,n\ge 2.$ The last statement on the invertibility when $\,n=d=1\,$ follows by the first
criterion of Proposition \ref{propM}.
\end{proof}

\begin{theorem}\label{theoremM2}
For $\,\beta> n/2,\,$ put
\begin{equation}
\phi_\beta(r) = (1+ r^2)^{-\beta} \qquad(r\ge 0).
\end{equation}
Let $X$ be an arbitrary set of points of $\R^n$ satisfying \eqref{M1} and
$$\mathbf{S}_X\left(\phi_\beta\right) = \Big[ \phi_\beta\left(\bx_j - \bx_k\right)\Big]_{j, \,k\in\mathbb{N}}.$$

\begin{itemize}
\item[\rm(i)] $\,\mathbf{S}_X\left(\phi_\beta\right)$ defines a bounded self-adjoint operator on $\ell^{2}(\mathbb{N})$ with
$$\left\| \mathbf{S}_X\left(\phi_\beta\right)\right\| \le 1 + \frac{d(5^d -1) B\left(\beta-\frac d2,\,\frac d2\right)}
{2 \left[\delta(X)\right]^d}\,.$$

\item[\rm(ii)] For $\,n\ge 2,\,$ $\,\mathbf{S}_X\left(\phi_\beta\right)$ defines a bounded invertible operator on $\ell^{2}(\mathbb{N})$.
In the case $\,n=d=1,\,$ if $X$ satisfies the additional assumption
$$\delta(X) > 2 B\left(\beta- \frac 12\,,\, \frac 12\right),$$
then it defines a bounded invertible operator on $\ell^{2}(\mathbb{N})$.
\end{itemize}
\end{theorem}

\begin{proof}
By using the aforementioned integral representation
$$\phi_\beta(r) = \frac{1}{\Gamma(\beta)}\int_0^\infty e^{-r^2 t} e^{-t} t^{\beta-1} dt\qquad(r\ge 0),$$
the proof follows along the same scheme as above.
\end{proof}

\begin{remark} In connection with the problem of interpolating functions at an arbitrary
set of distinct points $X$, it is an immediate consequence of Theorems \ref{theoremM1}, \ref{theoremM2} that
$\,\M_\alpha,\, \phi_\beta,\,$ with $\,\alpha>0,\,\beta>n/2,\,$ could be used
in constructing Lagrange-type radial basis sequences $\,\left\{u_j^*\right\}_{j\in\N},\,$ by the same process
pointed out in the introduction, and the interpolating functional
$$A_X(f)(\bx) =  \sum_{j=1}^\infty f(\bx_j)\,u_j^*(\bx).$$

\end{remark}

\section{Gramian matrices and Riesz sequences}
Now that Schoenberg matrices of Mat\'ern functions are shown to induce bounded and
invertible operators on $\ell^2(\N)$, it is natural to ask if they generate Riesz sequences
or bases in appropriate Hilbert spaces.

We recall that a system $\,\{f_j\}_{j\in\N}\,$ of vectors
in a Hilbert space $H$ is said to be a Riesz sequence if its moment space
is equal to $\ell^2(\N)$, that is,
$$\left\{ \mathbf{m}_f = \big\{(f, f_j)_H\big\}_{j\in\N} : f\in H\right\} = \ell^2(\N).$$
If  $\,\{f_j\}_{j\in\N}\,$ is complete in addition, it is called a Riesz basis (see \cite{Y}).
A classical theorem of Bari states a necessary and sufficient condition for
the system $\,\{f_j\}_{j\in\N}\,$ to be a Riesz sequence is that
the Gramian matrix
\begin{equation}
{\rm Gram}\Big(\{f_j\}_{j\in\N}\,;\,H\Big) = \big[\left( f_j,\,f_k\right)_H\big]_{j, \,k\in N}
\end{equation}
defines a bounded and invertible operators on $\ell^2(\N)$.

As for the sequences constructed from translating Mat\'ern functions by
distinct points, their Gramian matrices in $L^2(\R^n)$ or
Sobolev spaces turn out to be easily identifiable in terms of
Schoenberg matrices.

In order not to entangle with parameters, it is convenient to work with the Bessel
potential kernels of \eqref{G1}
\begin{equation*}
G_\alpha(\bx) = \frac{1}{2^{\alpha-1 + \frac n2}\,\pi^{\frac n2}\,
\Gamma(\alpha)}\,K_{\alpha-\frac n2}(|\bx|) |\bx|^{\alpha - \frac n2}.
\end{equation*}

\subsection{Results on $L^2(\R^n)$ space}
Concerning the square integrability, we have the following.

\begin{lemma}\label{lemmaGR1}
For $\,\lambda>-1,\,$ if $\,2\alpha + \lambda +1>0,\,$ then
\begin{equation*}
\int_0^\infty \big[K_\alpha(t) t^\alpha\big]^2 t^{2\lambda +1} dt = \frac{\sqrt{\pi}\,\,\Gamma(\alpha +\lambda +1)\Gamma(2\alpha +\lambda +1)
\Gamma(\lambda+1)}{4\,\Gamma\left(\alpha + \lambda + \frac 32\right)}\,.
\end{equation*}

In particular, if $\, \alpha + n/4>0\,$ with $n$ a positive integer, then
\begin{equation*}
\int_0^\infty \big[K_\alpha(t) t^\alpha\big]^2 t^{n-1} dt = \frac{\sqrt{\pi}\,\,\Gamma\left(\alpha +\frac n2\right)
\Gamma\left(2\alpha +\frac n2\right)\Gamma\left(\frac n2\right)}{4\,\Gamma\left(\alpha + \frac{n+1}{2}\right)}\,.
\end{equation*}
\end{lemma}

\begin{proof}
An application of Parseval's relation, Theorem \ref{Parseval}, for the Hankel-Schoenberg transforms
to formula \eqref{S7} of Theorem \ref{theoremS1} gives
\begin{align*}
\int_0^\infty \big[K_\alpha(t) t^\alpha\big]^2 t^{2\lambda +1} dt = \big[2^{\alpha+\lambda}\,\Gamma(\alpha +\lambda+1)\big]^2
\int_0^\infty \frac{r^{2\lambda+1}\,dr}{(1+ r^2)^{2\alpha+ 2\lambda+ 2}}.
\end{align*}
By making substitution $\, u = 1/(1+r^2),\,$ we compute
\begin{align*}
\int_0^\infty \frac{r^{2\lambda+1}\,dr}{(1+ r^2)^{2\alpha+ 2\lambda+ 2}}
&= \frac 12 \int_0^1 u^{2\alpha + \lambda} (1-u)^\lambda du\\
&= \frac 12\,B(2\alpha +\lambda+1,\,\lambda+1)
\end{align*}
and the stated formula follows on simplifying constants by using Legendre's duplication formula
for the Gamma function. The second stated formula corresponds to a special case of the first one with
$\,\lambda = n/2 -1.$
\end{proof}

\begin{theorem}\label{theoremGR1}
If $\,\alpha>n/4,\,$ then for any $\,\bx, \,\mathbf{y}\in\R^n,\,$
\begin{align}\label{GR1}
\big( G_{\alpha}(\cdot-\bx),\, G_\alpha(\cdot-\mathbf{y})\big)_{L^2(\R^n)}
= G_{2\alpha} (\bx-\mathbf{y}).
\end{align}
As a consequence, for any sequence of distinct points $\,(\bx_j)_{j\in\N}\subset \R^n,\,$
the Gramian matrix of the system $\,\big\{ G_\alpha(\bx- \bx_j) \big\}_{j\in\N}\subset
L^2(\R^n)\,$ coincides with the Schoenberg matrix of $G_{2\alpha}$, that is,
\begin{align*}
{\rm Gram}\Big(\big\{ G_\alpha(\bx- \bx_j)\big\}_{j\in\N}\,;\,L^2(\R^n)\Big)
= \Big[G_{2\alpha} \left(\bx_j- \bx_k\right)\Big]_{j, \,k\in\N}\,.
\end{align*}
\end{theorem}

\begin{proof}
By Lemma \ref{lemmaGR1}, $\,G_\alpha\in L^2(\R^n).\,$ Due to radial symmetry,
\begin{align*}
\big( G_\alpha(\cdot-\bx),\, G_\alpha(\cdot-\mathbf{y})\big)_{L^2(\R^n)}
&=\int_{\R^n} G_\alpha(\mathbf{u}-\bx) G_\alpha(\mathbf{u} -\mathbf{y}) d\mathbf{u}\\
&= \int_{\R^n} G_\alpha(\bx - \mathbf{y} -\mathbf{w}) G_\alpha(\mathbf{w}) d\mathbf{w}\\
&= \left(G_\alpha\ast G_\alpha\right)(\bx-\mathbf{y}).
\end{align*}
On the Fourier transform side, formula \eqref{S9} gives
\begin{align*}
\widehat{G_\alpha\ast G_\alpha}(\xi)
= (1+|\xi|^2)^{-2\alpha} = \widehat{G_{2\alpha}}(\xi),
\end{align*}
whence $\,G_\alpha \ast G_\alpha = G_{2\alpha}\,$ and the result follows.
\end{proof}

\begin{remark} This result extends the work of L. Golinskii {\it et al.} \cite{GMO} in which the
authors dealt only with the range $\,n/4<\alpha\le n/2.$

\begin{itemize}
\item[(a)] In the case when both $\,\alpha-n/2\,$ and $\,2\alpha - n/2\,$ are halves of odd integers,
 it is possible to write $L^2$ inner products explicitly with the aid of (K5).
 As illustrations in $\R^3$, we take $\,\alpha = 1, \,2\,$ to obtain
\begin{align*}
&\qquad\qquad\int_{\R^3} \frac{e^{-|\mathbf{u} -\bx| - |\mathbf{u}-\by|}}
{|\mathbf{u}-\bx|\,|\mathbf{u}-\by|}\,d\mathbf{u}
= 2\pi\,e^{-|\bx-\by|}\,,\\
&\int_{\R^3} e^{-|\mathbf{u} -\bx| - |\mathbf{u}-\by|}\,d\mathbf{u}
= \pi\,e^{-|\bx-\by|}\left( 1+ |\bx-\by| + \frac{|\bx-\by|^2}{3}\right)
\end{align*}
for which the first formula is of considerable interest in the spectral
analysis for the Schr\"odinger equations (see \cite{MS}).

\item[(b)] To reformulate \eqref{GR1} in a more direct fashion, put
\begin{equation}\label{GR2}
F_\alpha(\bx) = \frac{1}{2^{\alpha +n -1} \pi^{\frac n2}\Gamma\left(\alpha + \frac n2\right)}\,
K_\alpha(|\bx|) |\bx|^\alpha\,.
\end{equation}
As an alternative of \eqref{GR1}, if $\,\alpha>-n/4,\,$ then
\begin{align}
\big( F_\alpha(\cdot-\bx),\, F_\alpha(\cdot-\mathbf{y})\big)_{L^2(\R^n)}
= F_{2\alpha + \frac n2} (\bx-\mathbf{y}).
\end{align}
\end{itemize}
\end{remark}

As it is shown in Theorem \ref{theoremM1} that
the Schoenberg matrices of
$$G_{2\alpha}(z) = \frac{\Gamma(2\alpha-n/2)}{(4\pi)^{n/2}\,\Gamma(2\alpha)}\,
\M_{2\alpha -n/2}(z)\qquad(z>0)$$
define bounded and invertible
operators on $\ell^2(\N)$ as long as $\,\alpha>n/4,$ we obtain
the following from Bari's theorem and Theorem \ref{theoremGR1}.

\begin{corollary} Let $\,\alpha>n/4\,$ and
$\,X =\left\{\bx_j\in\R^n : j\in\N\right\}\,$ be arbitrary with
$$\delta(X) = \inf_{j\ne k}\,\left|\bx_j - \bx_k\right|\,>0.$$

\begin{itemize}
\item[\rm(i)] If $\,n\ge 2,\,$ then $\,\big\{ G_\alpha(\bx- \bx_j) \big\}_{j\in\N}\,$
forms a Riesz sequence in $L^2(\R^n).$
\item[\rm(ii)] In the case $\,n=1,\,$ if $X$ is separated with
$$\delta(X)> \frac{4\,\Gamma(2\alpha) \Gamma(1/2)}{\Gamma(2\alpha-1/2)},$$
then $\,\big\{ G_\alpha(x - x_j) \big\}_{j\in\N}\,$
forms a Riesz sequence in $L^2(\R).$
\end{itemize}
\end{corollary}

\subsection{Results on Sobolev spaces}
An important feature of the Sobolev space $H^\alpha(\R^n)$ with $\,\alpha>n/2\,$
is that it is a reproducing kernel Hilbert space with the kernel
$G_{\alpha}(\bx-\by)$ so that it may be viewed as the space of functions of type
$$f(\bx)= \sum_{j=1}^\infty a_j \,G_{\alpha}(\bx - \bx_j),$$
where $\,(a_j)\in\ell^2(\N)\,$ and $\,(\bx_j)\subset\R^n\,$ are arbitrary (see \cite{A}).
Thus it is reasonable to expect that the system $\,
\left\{G_{\alpha}(\bx - \bx_j)\right\}_{j\in\N}\subset H^\alpha(\R^n)\,$ may serve as a Riesz sequence
or a Riesz basis in its closed linear span once the translation points $(\bx_j)$ were scattered all
over some planes of $\R^n$.

As a matter of fact, the reproducing property implies
\begin{equation}\label{GR3}
\big( G_{\alpha}(\cdot-\bx),\, G_{\alpha}(\cdot-\by)\big)_{H^\alpha(\R^n)}
= G_{\alpha} (\bx-\by)
\end{equation}
for all $\,\bx, \,\by\in\R^n\,$ and our foregoing analysis yields

\begin{theorem}
Let $\,\alpha>n/2\,$ and $\,X =\left\{\bx_j\in\R^n : j\in\N\right\}\,$ be arbitrary with
$$\delta(X) = \inf_{j\ne k}\,\left|\bx_j - \bx_k\right|>0.$$

\begin{itemize}
\item[\rm(i)] If $\,n\ge 2,\,$ then $\,\big\{ G_{\alpha}(\bx- \bx_j) \big\}_{j\in\N}\,$
forms a Riesz sequence in $H^\alpha(\R^n).$
\item[\rm(ii)] In the case $\,n=1,\,$ if $X$ is separated with
$$\delta(X)> \frac{4\,\Gamma(\alpha) \Gamma(1/2)}{\Gamma(\alpha- 1/2)},$$
then $\,\big\{ G_{\alpha}(x - x_j) \big\}_{j\in\N}\,$
forms a Riesz sequence in $H^\alpha(\R).$
\end{itemize}
\end{theorem}

Regarding the problem of determining if the sequences of translates by inverse multi-quadrics
give rise to Riesz sequences, we introduce a class of function spaces defined in terms of Fourier transforms as follows.

\begin{definition}
For $\,\alpha>0,$
\begin{align*}
\mathcal{K}_\alpha(\R^n) = \left\{ f\in C(\R^n)\cap L^2(\R^n) :
\int_{\R^n} \frac{\big|\widehat f(\mathbf{\xi})\big|^2 d\mathbf{\xi}}
{ K_{\alpha} (|\xi|) |\xi|^{\alpha}} <\infty\right\}.
\end{align*}
\end{definition}

A theorem of R. Schaback \cite{S} and H. Wendland (\cite{We}, Theorem 10.27) states
if $\,\Phi\in C(\R^n)\cap L^1(\R^n),\,$ real-valued and positive definite,
the Hilbert space of functions on $\R^n$ with the reproducing kernel $\Phi(\mathbf{x}-\mathbf{y})$ coincides with
\begin{align*}
\mathcal{H}(\R^n) = \left\{ f\in C(\R^n)\cap L^2(\R^n) : \int_{\R^d} \frac{\big|\widehat f(\mathbf{\xi})\big|^2 d\mathbf{\xi}}
{\widehat{\Phi}(\mathbf{\xi})} <\infty\right\}
\end{align*}
for which the inner product is defined by
\begin{align*}
\bigl(f,\,g\bigr)_{\mathcal{H}(\R^n)} = (2\pi)^{-n}\int_{\R^n} \frac{\widehat{f}(\mathbf{\xi})
\overline{\,\widehat{g}(\mathbf{\xi})} \,d\mathbf{\xi}}{\widehat{\Phi}(\mathbf{\xi})}.
\end{align*}

As a consequence, it is simple to find that the space $\K_\alpha(\R^n)$ arises as
a reproducing kernel Hilbert space with an appropriate multi-quadrics as its reproducing kernel.
To be precise, we have the following results.

\begin{theorem}
For $\,\beta>n/2,\,$ consider the inverse multi-quadrics
$$\phi_\beta(\bx) = (1+|\bx|^2)^{-\beta}.$$

\begin{itemize}
\item[\rm(i)] The Hilbert space of functions on $\R^n$ with the reproducing kernel $\phi_\beta$
coincides with $\,\K_{\beta-n/2}(\R^n)\,$ for which the inner product is defined by
\begin{align}\label{GR4}
\bigl(f,\,g\bigr)_{\K_{\beta-n/2}(\R^n)} = (2\pi)^{-2n}\int_{\R^n} \frac{\widehat{f}(\mathbf{\xi})
\overline{\,\widehat{g}(\mathbf{\xi})} \,d\mathbf{\xi}}{G_\beta(\mathbf{\xi})}.
\end{align}

\item[\rm(ii)] Let $\,X =\left\{\bx_j\in\R^n : j\in\N\right\}\,$ be arbitrary with
$$\delta(X) = \inf_{j\ne k}\,\left|\bx_j - \bx_k\right|>0.$$
Then the system $\,\big\{ \phi_\beta(\bx- \bx_j) \big\}_{j\in\N}\,$
forms a Riesz sequence in $\K_{\beta-n/2}(\R^n)\,$ for any $\,n\ge 2\,$ and for $\,n=1\,$
under the additional assumption
$$\delta(X)> 2 B(\beta-1/2,\,1/2).$$
\end{itemize}
\end{theorem}

\begin{proof}
Obviously, $\phi_\beta$ is continuous, integrable and  positive definite.
By an application of the Hankel-Schoenberg transform formula for $\phi_\beta$ as stated in
Corollary \ref{corollaryS1}, we have
$\,\,\widehat{\phi_\beta}(\xi) = (2\pi)^n G_\beta(\xi)\,$
and hence part (i) follows
by the aforementioned theorem of Schaback and Wendland.

By the reproducing property, the Gramian matrix is given by
\begin{align*}
{\rm Gram}\Big(\big\{\phi_\beta(\bx- \bx_j)\big\}_{j\in\N}\,;\,\K_{\beta-n/2}(\R^n)\Big)
= \Big[\phi_{\beta} \left(\bx_j- \bx_k\right)\Big]_{j, \,k\in\N}\,.
\end{align*}
and part (ii) follows immediately from Theorem \ref{theoremM2}.
\end{proof}

\begin{remark}
As the Mat\'ern functions of positive order are bounded smooth functions with exponential decays,
it is evident $\,\K_\alpha(\R^n)\subset H^\infty(\R^n)\,$ for any $\,\alpha>0.$
In the special case $\,\beta= (n+1)/2,\,$ we note
$$\K_{1/2}(\R^n) = \left\{f\in C(\R^n)\cap L^2(\R^n) :
\int_{\R^n} e^{\,|\xi|} \,\big|\widehat f(\mathbf{\xi})\big|^2\, d\mathbf{\xi} <\infty\right\},
$$
which is the reproducing kernel Hilbert space with the Poisson kernel
$$\phi_{\frac{n+1}{2}}(\bx) = (1+|\bx|^2)^{-\frac{n+1}{2}}.$$
\end{remark}

\bigskip

\section{Appendix: $\ell^2(\N)$-Boundedness}

For the sake of completeness, we reproduce the proof of L. Golinskii {\it et al.} \cite{GMO}
for part (i) of Proposition \ref{propM} which states

\begin{itemize}{\it
\item[{}] Suppose that $f$ is a nonnegative monotone decreasing function on $[0, \infty)$ such that
$\,f(0) =1\,$ and the function $\,f(t) t^{d-1}\,$ is integrable on $[0, \infty).$
For any $\,X\subset \R^n\,$ satisfying the condition \eqref{M1}, the Schoenberg matrix
$\,\mathbf{S}_{X}(f)$ defines a bounded self-adjoint operator on $\ell^{2}(\mathbb{N})$ with
\begin{equation}\label{A}
\left\|\mathbf{S}_{X}(f)\right\|\le 1+\frac{ d(5^d-1)}{[\delta(X)]^d}\int_{0}^{\infty}f(t)t^{d-1}dt.
\end{equation}}
\end{itemize}

\medskip
\paragraph{Proof.}
Let us write $\,\delta = \delta(X)\,$ and assume $\,{\rm span}(X)\simeq\R^{d}\,$ for simplicity.
We fix $j$ and estimate the infinite sum
\begin{align*}
A_j &\equiv \sum_{k=1}^{\infty}f(|\bx_k- \bx_j|)=1+\sum_{m=1}^{\infty}\sum_{\bx_{k}\in X_{m}}f(|\bx_k- \bx_j|)\,,\quad\text{where}\\
X_{m} &=\big\{\bx_{k}\in X : m\delta \le |\bx_k- \bx_j|< (m+1)\delta\big\}\,.
\end{align*}
In terms of the open balls $\,B(\bx_k,\delta/2)\subset \R^n\,,$ a geometric inspection reveals
\begin{align*}
\#(X_{m}) &\le \frac{\,\mathrm{vol}\Big(\Big\{\by\in \R^d : \left(m-\frac{1}{2}\right)\delta\le |\by - \bx_j|<
\left(m+\frac{3}{2}\right)\delta\Big\}\Big)\,}
{\text{vol}\big(\,B(\bx_k,\delta/2)\cap \R^d\,\big)}\\
&=(2m+3)^d-(2m-1)^d\\
&\le (5^d-1)\,m^{d-1},
\end{align*}
which implies
\begin{equation*}
A_j \le  1+\sum_{m=1}^{\infty}(5^d-1)m^{d-1}f(m\delta)\,.
\end{equation*}

As $f$ is monotone decreasing on $[0,\infty)$,
\begin{align*}
\int_{0}^{\infty}f(t\delta)t^{d-1}dt &=\sum_{m=1}^{\infty}\int_{m-1}^{m}f(t\delta)t^{d-1}dt\\
&\ge \sum_{m=1}^{\infty} f(m\delta)\left[\frac{m^{d}-(m-1)^d}{d}\right]\\
&\ge\sum_{m=1}^{\infty} f(m\delta)\frac{m^{d-1}}{d},
\end{align*}
which yields
$$ \sum_{m=1}^{\infty}f(m\delta)m^{d-1}\le \frac{d}{\delta^{d}}\int_{0}^{\infty}f(t)t^{d-1}dt.$$
Inserting this estimate into the above sum, we are led to
\begin{align*}
A_j \le 1+\frac{d(5^d-1)}{\delta^d}\int_{0}^{\infty}f(t)t^{d-1}dt\,.
\end{align*}

Since this estimate is independent of $j$, the result follows by Schur's test.

\begin{remark}
By Schur's test, (\ref{A}) implies
\begin{equation*}
\left\|I-S_X(f)\right\|\le \frac{d(5^d-1)}{[\delta(X)]^d}\int_{0}^{\infty}f(t)t^{d-1}dt
\end{equation*}
and the right side is strictly less than $1$ if
\begin{equation*}
\delta(X)>\left[d(5^d-1)\int_{0}^{\infty}f(t)t^{d-1}dt\right]^{1/d}\,.
\end{equation*}
For such a set $X$, $S_X(f)$ defines a bounded invertible operator on $\ell^{2}(\mathbb{N})$.
\end{remark}
\bigskip

\noindent
{\bf Acknowledgements.} Yong-Kum Cho is supported by National Research Foundation of Korea Grant
funded by the Korean Government (\# 20150301). Hera Yun is supported by the Chung-Ang University
Research Scholarship Grants in 2014.

\bigskip
\noindent
Yong-Kum Cho

\noindent
Department of Mathematics, Chung-Ang University, 84 Heukseok-Ro, Dongjak-Gu, Seoul 156-756, Korea (e-mail: ykcho@cau.ac.kr)

\bigskip
\noindent
Dohie Kim

\noindent
Department of Mathematics, Chung-Ang University, 84 Heukseok-Ro, Dongjak-Gu, Seoul 156-756, Korea (e-mail: hanna927@hanmail.net)

\bigskip
\noindent
Kyungwon Park

\noindent
Department of Computer Engineering, Korea Polytechnic University,
237 Sangidaehak-Ro, Siheung-Si 429-793, Korea (e-mail: chrisndanny@kpu.ac.kr)

\bigskip
\noindent
Hera Yun

\noindent
Department of Mathematics, Chung-Ang University,
84 Heukseok-Ro, Dongjak-Gu, Seoul 156-756, Korea (e-mail: herayun06@gmail.com)

\end{document}